\documentclass[11pt, reqno]{amsart}
\usepackage{amsmath,amssymb,amsbsy,amsfonts,amsthm,latexsym,amsopn,amstext,amsxtra,epic, euscript,amscd,indentfirst,verbatim,graphicx, hyperref,xcolor, setspace,tikz,tikz-cd}
\usepackage[normalem]{ulem} 
\usepackage{tabularx} 
\usepackage{relsize} 
\usepackage[margin=1.2in]{geometry}
\usepackage{enumitem} 
\usepackage{hyperref} 
\newcommand{\stkout}[1]{\ifmmode\text{\sout{\ensuremath{#1}}}\else\sout{#1}\fi} 
\DeclareMathOperator{\Tr}{Tr}
\DeclareMathOperator{\Ric}{Ric}

\DeclareMathOperator{\Hess}{Hess}

\begin{document}

\numberwithin{equation}{section}
\newtheorem{theorem}{Theorem}[section]
\newtheorem*{theorem*}{Theorem}
\newtheorem{theoremstar}[theorem]{Theorem*}
\newtheorem{conjecture}[theorem]{Conjecture}
\newtheorem*{conjecture*}{Conjecture}
\newtheorem{proposition}[theorem]{Proposition}
\newtheorem*{proposition*}{Proposition}
\newtheorem*{``proposition"*}{``Proposition"}
\newtheorem{question}{Question}
\newtheorem{lemma}[theorem]{Lemma}
\newtheorem*{lemma*}{Lemma}
\newtheorem{cor}[theorem]{Corollary}
\newtheorem*{obs*}{Observation}
\newtheorem{obs}{Observation}
\newtheorem{example}[theorem]{Example}
\newtheorem{condition}{Condition}
\newtheorem{definition}[theorem]{Definition}
\newtheorem*{definition*}{Definition}
\newtheorem{proc}[theorem]{Procedure}
\newtheorem{problem}{Problem}
\newtheorem{remark}{Remark}
\newcommand{\comments}[1]{} 
\def\Z{\mathbb Z}
\def\Za{\mathbb Z^\ast}
\def\Fq{{\mathbb F}_q}
\def\R{\mathbb R}
\def\N{\mathbb N}
\def\C{\mathbb C}
\def\k{\kappa}
\def\grad{\nabla}
\def\M{\mathcal{M}}
\def\S{\mathcal{S}}
\def\pt{\partial}

\newcommand{\todo}[1]{\textbf{\textcolor{red}{[To Do: #1]}}}
\newcommand{\note}[1]{\textbf{\textcolor{blue}{#1}}}

\title[]{A Priori Log-Concavity Estimates for Dirichlet Eigenfunctions}
 \author[Iowa State University]{Gabriel Khan} 
 \address[Gabriel Khan]{Department of Mathematics, Iowa State University, Ames, IA, USA.}
  \email{gkhan@iastate.edu}  

  \author[University of Strasbourg]{Soumyajit Saha} 
\address[Soumyajit Saha]{IRMA, University of Strasbourg, Strasbourg, France.}  \email{soumyajit.saha@unistra.fr}

   \author[UCSB]{Malik Tuerkoen}
   \address[Malik Tuerkoen]{Department of Mathematics, University of California,  Santa Barbara, CA, USA.}
  \email{mmtuerkoen@ucsb.edu}

\date{\today}

\allowdisplaybreaks

\begin{abstract}
In this paper, we establish a priori log-concavity estimates for the first Dirichlet eigenfunction of convex domains of a Riemannian manifold. Specifically, we focus on cases where the principal eigenfunction $u$ is assumed to be log-concave and our primary goal is to obtain quantitative estimates for the Hessian of $\log u$.
\end{abstract}

\maketitle 

\section{Introduction}

Consider a bounded, convex domain $\Omega$ within a Riemannian manifold $(M, g)$. We study the properties of the Dirichlet groundstate eigenfunction of the Laplace-Beltrami operator. More precisely, we consider positive $L^\infty$-normalized solutions to the problem
\begin{equation}\label{eqn: Eigenfunction equation}
   \begin{cases}
       -\Delta u  = \lambda u \quad &\textup{in }\Omega,\\
       u = 0 \quad &\textup{on }\partial \Omega.
   \end{cases}
\end{equation}

For convex domains in $\mathbb{R}^n$ \cite{brascamp1976extensions}, $\mathbb{S}^n$ \cite{lee1987estimate}, and $C^4$-small deformations of $\mathbb{S}^2$ \cite{surfacepaper1,khan2023modulus}, the principal eigenfunction $u_1$ will be log-concave (i.e., $v = \log u_1$ is concave). A large amount of work has been done studying the log-concavity (or lack thereof) for the principal eigenfunction for various classes of domains and geometries \cite{singer1985estimate, wang2000estimation,  bourni2022vanishing, 10.4310/jdg/1559786428,Ishige2022,khan2022negative}. It is well known that the log-concavity of the ground state implies lower bounds on the spectral gap (or {fundamental gap})—the difference between the two lowest eigenvalues of the Laplace operator with Dirichlet bondary conditions. Most notably, Andrews and Clutterbuck \cite{andrews2011proof} established a strong form of log-concavity in order to prove the fundamental gap conjecture. 

In this paper, we focus on cases where the principal eigenfunction is log-concave, aiming to derive stronger estimates for its Hessian. This approach is roughly analogous to establishing a priori estimates for solutions to Monge-Amp\`ere equations or other fully-nonlinear degenerate elliptic equations. In that setting, one assumes that a convex solution exists and attempts to derive a $C^2$ (or similar) estimate for it, known as an \emph{a priori estimate}. For this reason, we can interpret the results in this paper as providing \emph{a priori log-concavity estimates}. In the setting of Monge-Amp\`ere equations, the a priori estimate can subsequently be combined with a continuity argument to establish the existence for the original problem. Similarly, the ultimate goal for our work is to establish the log-convexity of the principal eigenfunction for a broader class of geometries.

\subsection{Main results}
\begin{theorem}{} \label{Log-concavity-estimate-sphere}
Suppose that $(M^n,g)$ is a Riemannian manifold with bounded 
sectional curvature $\overline \kappa \geq \kappa \geq \underline \kappa.$ 
 Then for any uniformly convex $C^2$ domain $\Omega \subset M$, if the positive solution of the problem \eqref{eqn: Eigenfunction equation} is log-concave, then $v = \log u$ satisfies
\begin{equation}\label{Main inequality}
    \nabla^2 v + \left(\alpha | \nabla \sqrt u  |^2 + C v -d\right) g < 0,
\end{equation}
where $\nabla^2 v$ denotes the Hessian tensor and $\alpha$,  $C$ and $d \geq 0$ are explicit constants depending on the principal eigenvalue of \eqref{eqn: Eigenfunction equation}, $\underline \kappa$, $\overline \kappa$, $|\nabla \Ric|$ and a lower bound on the second fundamental form of the boundary of $\Omega$.   
\end{theorem}

Note that \eqref{Main inequality} is not scale-invariant, which is why we only consider $L^\infty$-normalized eigenfunctions (i.e., eigenfunctions satisfying $\sup u = 1$).
There are several special cases where we can solve the constants in closed form and that we know that the principal eigenfunction is log-concave.\footnote{In the following, we use the notation $\nabla^2 v + b < 0$ with $b$ a function to indicate that the two-tensor $\nabla^2 v + b g$ is negative definite.}

\begin{cor} \label{Spherical version of Bochner theorem}
     Suppose that $\Omega \subset \mathbb{S}^n$ is a bounded domain whose boundary satisfies $\textrm{II} (\partial \Omega) > \frac{1}{3}$, where $\textrm{II} (\partial \Omega)$ is the second fundamental form of the boundary of $\Omega$. Then the principal eigenfunction of the 
 Dirichlet Laplacian satisfies the estimate
     \begin{equation*}
         \nabla^2 v + \frac{1}{\lambda} \left( \sqrt{ (\lambda+2)^2 + 8 \lambda} - (\lambda+2) \right)|\nabla \sqrt u|^2+v<0.
     \end{equation*}    
\end{cor}

 An analogous result holds for convex domains in Euclidean space. 

\begin{cor} \label{Euclidean version of Bochner theorem}
     Suppose that $\Omega \subset \mathbb{R}^n$ is a bounded uniformly convex domain. Then the positive solutions $u$ of \eqref{eqn: Eigenfunction equation} satisfies the estimate
     \begin{equation*}
         \nabla^2 v +  \alpha|\nabla \sqrt u|^2+v- \sqrt{\tfrac{\lambda}{2}}<0,
     \end{equation*}
      where $\alpha$ satisfies
      \begin{equation}
        \alpha \leq \frac{1}{\lambda} \left(\sqrt{ \lambda^2 + 8 \lambda} - \lambda \right) \textrm{ and } \alpha < \inf_{X \in U\partial \Omega}  \frac{4 \textrm{II}(X,X)}{|\nabla u(x)|}.   \end{equation}
\end{cor}

 It should be noted that this estimate is somewhat weaker compared to Corollary \ref{Spherical version of Bochner theorem} due to the curvature vanishing. In the proof of Theorem \ref{Log-concavity-estimate-sphere}, the positive curvature is used to control certain terms. Without these terms, it is necessary to include the additional $-\sqrt{\frac{\lambda}{2}}$.
 The main ingredient in the proofs of our theorems is a maximum principle. For this, we construct a one-parameter family of barriers $b(x,t)$ (indexed by time $t$), which depend smoothly on $t$, such that $b(x, 1)$ is the barrier of interest and $b(x, 0)$ is chosen such that $\nabla^2 \log u + b(\cdot , 0) <0$.  By means of the maximum principle, one shows that the estimate holds through all times $t.$ More precisely, we consider barrier functions $b(x,t)$ and derive a criteria which prevents the Hessian of log-eigenfunctions from ever \emph{touching} the barrier, see Section \ref{section: prelim} for more details.

\subsection{Remarks about these results}
One natural question is whether the estimate \eqref{Main inequality} is sharp. For smooth domains, the Hopf lemma implies that $\alpha|\nabla \sqrt u|^2$ tends to infinity at the rate of $O\left(u^{-1}\right)$ near the boundary. Thus \eqref{Main inequality} provides a lower bound for how quickly $\nabla^2 \log u$ tends to negative infinity near the boundary.  As shown in Lemma \ref{Lemma at boundary, second attempt} (cf. Lemma 4.2 \cite{andrews2011proof}), this estimate is sharp (up to a multiplicative factor) for $\nabla^2 \log u$.

Finally, we note that Theorem \ref{Log-concavity-estimate-sphere} can be extended to more general problems, in particular to the equation
\begin{equation*}
   \begin{cases}
       -\Delta u + V u = \lambda \rho u, \\
       u \vert_{\partial \Omega} = 0
   \end{cases}
\end{equation*}
 where $V$ is not too concave and $\rho$ is not too convex (see \cite{khan2024concavity, khan2024spectral} for more detail on this equation and its applications to conformal geometry).

\subsection*{Acknowledgements}

The authors would like to thank Julie Clutterbuck for informing us of the conjecture in \cite{ishige2020new}. 
G. Khan was supported in part by Simons Collaboration Grant 849022.
This work was supported by the National Science
Foundation under Grant No. DMS-1928930, while the third named author was in
residence at the Simons Laufer Mathematical Sciences Institute
(formerly MSRI) in Berkeley, California, during the Fall 2024 semester.

\section{Preliminaries}\label{section: prelim}

 For a Riemannian manifold $(M,g)$, we write the Riemannian curvature $(3,1)$ tensor of its Levi-Civita connection as \[
R(X,Y)Z = \nabla_X \nabla_Y Z -  \nabla_Y \nabla_X Z - \nabla_{[X,Y]} Z. \]
We let $\Ric$ denote the $(0,2)$ Ricci tensor and 
and define $R_X$ to be the $(1,1)$-tensor given by \[
R_X (Y) = R(Y,X) X. 
\] 
Similar to \cite{surfacepaper1,khan2024concavity}, we introduce the notion of the \emph{barrier operator,} which is crucial in our approach to establish concavity estimates.

\begin{definition} \label{Barrier operator definition}
    Given a function $b: \Omega \rightarrow \mathbb R$ and a unit vector $X$, the barrier operator $\mathcal{B}$ is the quantity
\begin{align} \label{Barrier operator definition equation}
    \mathcal{B}(b,X) :=& -2b^2 +2\langle \nabla b, \nabla v\rangle -2\textup{tr}\left( R_{X}\circ (\nabla v \otimes \nabla v + \nabla ^2 v)\right) -2b\Ric({X,X}) 
    \\
& \nonumber \quad -\nabla _{\nabla v}\Ric(X,X)+2\nabla _{X}\Ric(X, \nabla v)+\Delta b, 
\end{align}
where $\lambda$ denotes the principal eigenvalue of Equation \ref{eqn: Eigenfunction equation} and $ v= \log u$, the logarithm of the principal eigenfunction of \eqref{eqn: Eigenfunction equation}.  
\end{definition}

This operator appears naturally in a certain maximum principle computation. More precisely, suppose one has a unit vector $X_p $ such that \eqref{map} is maximized and equals $0,$ then one finds that $\Delta \left( \nabla ^2 v (X,X) + b \right)(p) = \mathcal B(b,X)(p). $ Since we assume this is a maximum, this expression has a sign. To arrive at a contradiction, one needs to verify that this expression is of the opposite sign  and so we define a condition known as the \emph{barrier criteria}.\footnote{This computation was done carefully in \cite{surfacepaper1, khan2024concavity}.}

\begin{definition} \label{Barrier criteria definition}
A barrier function $b$ satisfies the \emph{barrier criteria} if $\mathcal{B}(b,X) >0$
    whenever $X \in U \Omega$ is a unit vector\footnote{Here, $U\Omega$ is the unit tangent bundle of $\Omega$.} such that the mapping
\begin{align}
 U\Omega \rightarrow \mathbb R, \quad X_q\mapsto \nabla ^2v(X_q,X_q) +b(q)  \label{map}
\end{align} 
achieves a maximum at $X$ and satisfies $\nabla ^2 v (X_q,X_q) + b(q)=0$. 
\end{definition}

The following proposition can be summarized as follows: If a function satisfies the barrier criteria (and other technical properties), then the eigenvalues of $\nabla^2 v$ will never touch the barrier, and thus the corresponding upper bound on  $\nabla^2 v$ will be preserved along the continuity family. In this paper, we only consider continuity families that deform the barrier function. Nevertheless, we state a more general version, since the proof is exactly the same.

\begin{proposition}
\label{mainpropposition}
Let $M^n$ be a smooth manifold and consider a one-parameter family of eigenvalue problems
\begin{equation}\label{ean: parametric eigenvalue equation}
      \begin{cases}
       -\Delta_{g(t)} \varphi + \varphi = \lambda \varphi \\
       \varphi \vert_{\partial \Omega(t)} = 0.
   \end{cases}
\end{equation}
Here, the metric $g(t)$ and the domain $\Omega(t)$ are allowed to depend on $t$, so long as we assume that the domain $\Omega(t)$ is geodesically convex with respect to $g(t)$ for all $t$. Let $v (t)= \log u (t)$, (omitting the dependence on $x$ for notational purposes) and $u(t)$ is a positive eigenfunction of equation \eqref{ean: parametric eigenvalue equation} on $\Omega(t)$. 
Suppose there is a function $b:\Omega(t) \to \mathbb{R}$ which satisfies the following assumptions:
\begin{enumerate}
    \item $b$ depends smoothly on both $x$ and $t$.
    \item For all $t$, the barrier $b$ satisfies the growth rate condition
\begin{equation}\label{Growth rate assumption}
    \liminf_{x \to \partial{\Omega}, X \in U_x \Omega} \frac{- b(x,t)} { \nabla^2 v(t)(X,X)  } <1.
\end{equation}
\item At time $t=0$, \begin{equation} \label{Initial time assumption}
    \nabla^2 v(0) + b(0)\,  g(0) <  0.
\end{equation}
\item And finally, for all $0 \leq t \leq 1$, $b$
satisfies the barrier criteria.
\end{enumerate} 

Then the function $v(1)$ satisfies the concavity estimate 
\begin{equation} \nabla^2 v(1) + b(1)\,  g \le  0   \label{concave v} \end{equation} 
on the original domain $\Omega(1)$. 
\end{proposition}

The proof that there can be no interior points where the barrier touches the Hessian of $v$ is exactly the same as the proof in \cite{surfacepaper1}, and follows from the maximum principle. The main difficulty of that proof is to show that $\Delta \left(\nabla ^2 v(X,X) + b\right)(p) = \mathcal B(b,X)(p)$ at $X_p$ from Definition \ref{Barrier criteria definition}, we refer the reader to \cite{surfacepaper1} for a very detailed proof on this. However, there is one difference, that is we allow the barrier function to go to infinity near the boundary. For this reason, we must analyze the behavior at the boundary more carefully. We now prove a lemma about the behavior of $\nabla^2 v$ near the boundary of $\Omega$, which translates Equation \eqref{Growth rate assumption} into a quantitative version can be directly checked. We denote $u$ to be the positive solution of the problem \eqref{eqn: Eigenfunction equation} and $v = \log u.$

\begin{lemma} \label{Lemma at boundary, second attempt}
Suppose that $\partial \Omega$ is of class $C^2$ and $x_0 \in \partial \Omega$ is a point on the boundary where the domain is strictly convex. Whenever $b(x)$ satisfies the bound 
\begin{equation} \label{Assumption near x naught}
    \limsup_{x \to x_0} b(x) u (x) < \left( \inf_{X \in U \partial \Omega(x_0)} \textrm{II}_{x_0}(X,X) \right)| \nabla u(x_0)|,
\end{equation}
there is a small neighborhood around $x_0$ such that
\begin{equation}\label{Boundary inequality}
     \nabla^2 v + b g < 0.
\end{equation}
\end{lemma}

This is a modified version of Lemma 4.2 of \cite{andrews2011proof} (or Lemma 3.4 in \cite{10.4310/jdg/1559786428}), and the proof is very similar. The main difference is that we incorporate a barrier function $b$ to find the optimal growth rate at each boundary point. 
\begin{proof}
We start by estimating the quantity
$\nabla^2 u(X, X)$ at $x_0$ for $X\in U\partial \Omega(x_0),$ 
 and then estimate it in a neighborhood. Since $x_0 \in \partial \Omega,$ we have that $\nabla u=-\left.|\nabla u| 
~ \nu\right|_{x_0},$  where $\nu$ is the outward normal vector. Thus, since $X\in T\partial \Omega(x_0)$ and thus $X \perp \nu,$ and therefore 
\begin{align*}
\nabla^2 u(X, X) & =\left\langle\nabla_{X} \nabla u, X\right\rangle=-\left\langle\nabla_{X}(|\nabla u| \nu), X\right\rangle \nonumber =-\mathrm{II}(X, X) |\nabla  u(x_0)|\leq -k |\nabla u(x_0)|, \label{tangential-direction}
    \end{align*}
 where II is the second fundamental form of $\partial \Omega$ at $x_0$ and $k = \inf _{X\in U\partial \Omega (x_0)}\mathrm{II}(X,X) >0,$ by assumption.

To estimate $\nabla ^2 u(X,X)$ for $X \in U\Omega$ in a neighborhood around $x_0$, we consider the gradient direction $e=\frac{\nabla u}{|\nabla u|}$, which will be smooth near $x_0$ and denote the projection $\pi^\perp: X \mapsto\langle X, e\rangle e$, along with its orthogonal projection $\pi=\operatorname{id}-\pi^\perp$.
 Furthermore, since the domain is $C^2,$ the Hopf lemma implies that $|\nabla u(x_0)| > 0$ (see \cite{gilbarg1977elliptic}, Lemma 3.4).  Moreover, elliptic regularity implies that $u\in C^2(\overline \Omega).$ Thus, for any $\delta >0$, there exists an $r_0>0$ so that for $x \in B_{r_0}\left(x_0\right) \cap  \Omega$,
\begin{eqnarray}
\left.\nabla^2 u\right|_x(\pi X, \pi X) & \leq&\left( -k |\nabla u(x_0)|  + \delta \right) |\pi X|^2 \text { for any } X \in T_x \Omega,  \label{piX-hessian}\\
|\nabla u(x)| & \geq&\frac{ \left|\nabla u\left(x_0\right)\right|}{2},  \label{first-derivative-estimate}\\
0<u(x) & \leq& 2\left|\nabla u\left(x_0\right)\right| d\left(x, x_0\right) \label{u leq dist}.
    \end{eqnarray}

In this neighborhood, for any tangent vector $X$, we have that, using \eqref{piX-hessian}
\begin{align*}
  \nabla^2 u(X, X) 
& =\nabla^2 u(\pi X, \pi X)+2 \nabla^2 u\left(\pi X, \pi^{\perp} X\right)+\nabla^2 u\left(\pi^{\perp} X, \pi^{\perp} X\right) \\
& \leq\left( -k  |\nabla u(x_0)|  +\delta \right)|\pi X|^2+2 |\nabla ^2 u|_\infty|\pi X|\left|\pi^{\perp} X\right|+|\nabla ^2 u|_\infty\left|\pi^{\perp} X\right|^2 \\
& \leq\left( -k |\nabla u(x_0)|  +2\delta\right)|\pi X|^2+\left(|\nabla ^2 u|_\infty+\frac{|\nabla ^2 u|_\infty^2}{\delta}\right)\left|\pi^{\perp} X\right|^2, 
\end{align*}
where we have used Young's inequality in the final line. We can then use this to estimate the Hessian of $v$
\begin{eqnarray*}
\left.\nabla^2 v \right|_x(X, X) & =&\frac{1}{u}\left(\nabla^2 u(X, X)-\frac{\left(\nabla_{X} u\right)^2}{u}\right) \\
& \leq& \frac{1}{u}\left[ \left(-k |\nabla u(x_0)|+2\delta \right)|\pi X|^2  + \left(|\nabla^2u|_\infty+\frac{ |\nabla^2u|_\infty^2}{\delta} -\frac{|\nabla u(x_0)|}{8 d(x,x_0)} \right) \left|\pi^{\perp} X\right|^2 \right],
\end{eqnarray*}
where we used  \eqref{first-derivative-estimate} and \eqref{u leq dist} for the last term. 

In order to finish the proof of the lemma, using $g(X,X) =  |\pi X|^2 + |\pi^{\perp} X|^2= 1,$ we get that 
\begin{eqnarray*}
\nabla^2 v(X, X) + b 
\leq \frac{1}{u}\left(  \begin{aligned}
\left(-k |\nabla u(x_0)|+2\delta + b u \right)|\pi X|^2 \\  + \left(|\nabla^2u|_\infty+\frac{ |\nabla^2u|_\infty^2}{\delta} +b u -\frac{|\nabla u(x_0)|}{8 d(x,x_0)} \right) \left|\pi^{\perp} X\right|^2 \end{aligned} \right),
\end{eqnarray*}
and show that this is negative. Using the strict inequality in Assumption \eqref{Assumption near x naught}, we can find a $ \delta$ small enough such that  
$\left(-k |\nabla u(x_0)|+2\delta + b u \right) < 0,$ for $x\in B_{r_0'}(x_0)\cap \Omega$ and $r_0'\leq r_0$ sufficiently small.

For this choice of $\delta>0$, and since $b u$ is bounded in a neighborhood of $x_0,$ we can choose $r_0'$ small enough such that 
\[|\nabla^2u|_\infty+\frac{|\nabla^2u|_\infty^2}{\delta} +b u -\frac{|\nabla u(x_0)|}{8 d(x,x_0)} < 0 \quad \textup{in }B_{r_0'}(x_0)\cap  \Omega. \] 
\end{proof}

\begin{remark}
  This lemma does not provide quantitative bounds away from the boundary. 
Nonetheless, it is useful because it provides the precise limit for the growth of a valid barrier function. In other words, so long as $b$ grows no faster than \eqref{Assumption near x naught}, we can apply Proposition \ref{mainpropposition} to establish log-concavity results. It should also be noted that this condition is sharp, in that if there is a single point $x_0$ where \eqref{Assumption near x naught} is reversed,
then it is possible to find unit vectors so that inequality \ref{Boundary inequality} is reversed.

\end{remark}
\subsection{A simple illustration}

Before proving the main result, let us provide a simple example to show how a priori estimates can be used to strengthen log-concavity estimates. The computation of the barrier criteria is much simpler compared to those later in the paper, but it gives some flavor of how these arguments work. 
\begin{proposition}\label{prop: zero barrier}
 Suppose that $\Omega \subset (M,g)$ is a convex domain in a Einstein manifold with sectional curvature $\overline \kappa \geq \kappa \geq \underline \kappa > 0$ and $u$ is 
 log-concave. Then $v = \log u$ is a strictly concave function (i.e., has non-degenerate Hessian) in the region where $u > \sqrt{1 - \frac{\underline \kappa}{\overline \kappa}}$.  
\end{proposition}

Note that we assume that the principal eigenfunction is log-concave, since it remains an open problem to show that this holds for general convex domains within positively curved Einstein manifolds (other than the sphere). 

\begin{proof}
Consider the barrier function $b \equiv 0.$ Given a unit vector $X$, such as in Definition \ref{Barrier criteria definition} extend this vector to an orthonormal basis $\{e_j\}$ of $T_p\Omega$ which diagonalizes $R_X = (Y \mapsto R(X,Y)X)$. Using these coordinates and the assumption $\nabla \Ric \equiv 0$, for $b \equiv 0$ \eqref{Barrier operator definition} simplifies to
\begin{eqnarray}
    \mathcal{B}(b,X)= -2\textup{tr}\left( R_{X}\circ (\nabla v \otimes \nabla v + \nabla ^2 v)\right) =  -2\sum_{i=2}^n \kappa_i(v_i^2+v_{ii}),
\end{eqnarray}
where $\kappa_i = \langle R(e_1, e_i, e_1), e_i\rangle$ are sectional curvatures and we write $v_ i = \langle \nabla v, e_i\rangle$ and $v_{ii} = \nabla ^2 v(e_i,e_i)$.  Furthermore, since $\Delta v + |\nabla v|^2 = -\lambda,$
 \begin{align}
    \mathcal{B}(b,X) = -2\sum_{i = 2}^n \kappa_i(v_i^2+v_{ii})
      & =-2\sum_{i = 2}^n \underline\kappa(v_i^2+v_{ii})  -2\sum_{i = 2}^n (\kappa_i-\underline\kappa)(v_i^2+v_{ii}) \nonumber\\
      & \geq  -2\underline{\kappa}{}( -\lambda - v_1^2 - v_{11})  -2\sum_{i = 2}^n (\kappa_i-\underline\kappa)v_i^2 \nonumber \\
      &  \geq  2\underline{\kappa}{} \lambda -2(\overline \kappa - 
      \underline \kappa)|\nabla v|^2, \label{ineq: Prop}
 \end{align} 
 where we used $ 0 = v_{11} \geq v_{ii}$ for all $i \geq 2,$ in the first inequality. Then the barrier criterion is satisfied if $\lambda > (\tfrac{\overline \kappa}{\underline \kappa} -1) |\nabla v|^2$.

Using Lemma 2 from \cite{ling2006lower}, we get that
\begin{equation} \label{Gradient bound from Ling}
    \frac{|\nabla u |^2}{\beta^2-u^2} \leq \lambda 
\end{equation}
where $\beta>1$ is an arbitrary constant. In other words, we have that
\begin{equation}
\left(\tfrac{\overline \kappa}{\underline \kappa} -1 \right) |\nabla v|^2  =   \left(\tfrac{\overline \kappa}{\underline \kappa} -1\right)\frac{ |\nabla u|^2}{u^2} \leq \left(\tfrac{\overline \kappa}{\underline \kappa} -1\right)\frac{\lambda (\beta^2-u^2)}{u^2}. 
    \end{equation}
From which we may infer that $\lambda > \left(\tfrac{\overline \kappa}{\underline \kappa} -1\right) |\nabla v|^2$  is satisfied whenever
    \begin{equation}\label{Lower bound on u}
\left(\tfrac{\overline \kappa}{\underline \kappa} -1\right)\frac{\lambda (\beta^2-u^2)}{u^2} < \lambda.        
    \end{equation}
Letting $\beta$ go to $1$ in order to optimize the inequality in \eqref{Lower bound on u}, we find that the inequality holds whenever $u^2 > \left(\tfrac{\overline \kappa}{\underline \kappa} -1\right)-\left(\tfrac{\overline \kappa}{\underline \kappa} -1\right)u^2,$ or equivalently, whenever $u > \sqrt{1 - \frac{\underline \kappa}{\overline \kappa}}.$
\end{proof}

In fact, we can also show the following.
\begin{obs} \label{Nearly round spheres log-concavity}
    Suppose that $\Omega \subset (M^n,g)$ is a convex domain which satisfies $\underline{\kappa} = 1$. If the principle eigenfunction is log-concave, it must be strongly log-concave whenever 
    \begin{equation} \label{Nearly round spheres gradient bound}
        |\nabla v| <\frac{1}{4 (\overline{\kappa}- \underline{\kappa})}  \left( \sqrt{9 |\nabla \Ric|^2 + 16 (\overline{\kappa}-\underline{\kappa}) \underline{\kappa} \lambda } - 3 |\nabla \Ric| \right).
    \end{equation}
\end{obs}
\begin{proof}
    Similar to the computation above, for $b\equiv 0$ we find that
    \begin{eqnarray*}
        \mathcal B(b,X) &=& -2\textup{tr}\left( R_{X}\circ (\nabla v \otimes \nabla v + \nabla ^2 v)\right) -\nabla _{\nabla v}\Ric(X,X)+2\nabla _{X}\Ric(X, \nabla v) \\
        & > &  2\underline{\kappa}{} \lambda -2(\overline \kappa - 
      \underline \kappa)|\nabla v|^2 - 3 |\nabla \Ric| | \nabla v|,
    \end{eqnarray*}
    which is positive whenever \eqref{Nearly round spheres gradient bound} holds.
\end{proof}
Therefore, if the curvature pinching and derivative of the Ricci tensor go to zero, the region where weak log-concavity implies strong log-concavity becomes larger.

\section{Proof of Theorem \ref{Log-concavity-estimate-sphere}}

\label{Bochner Log-concavity estimates section}

We now prove the main result, and to do so we construct a family of functions which satisfy the barrier condition and tend to infinity near the boundary.

 \begin{lemma} \label{Lemma to show barrier property}
    Under the assumption of Theorem \ref{Log-concavity-estimate-sphere}, 
for $\alpha$ sufficiently small, $C$  and $d \geq 0$,  specified in the proof
    \begin{equation} \label{The barrier involving w}
     b(x,t) = \alpha |\nabla \sqrt{u}|^2 +Cv-d.
    \end{equation}
    satisfies the barrier criteria on the set where $b(x,t)>0$.
    \end{lemma}

\begin{proof}
Recall that $v = \log u$ and denote $w = \sqrt{u}$, and observe that in $\Omega$
\begin{eqnarray}
\label{Delta-w}
    \Delta w + \frac{ |\nabla w|^2}{w}=  - \frac{1}{2} \lambda w \quad \textup{and }\quad
     \Delta v +|\nabla v|^2 = -\lambda. 
\end{eqnarray}
We let $X_p\in T_p\Omega$ be a unit vector so that $\nabla ^2v({X,X}) +b(p)=0$
and suppose that $X_p$ maximizes the Hessian of $v$ at $p$. We denote $e_1=X_p$, 
and extend this vector to an orthonormal basis $\{e_j\}$ of $T_p\Omega$ which diagonalizes $R_X = (Y \mapsto R(X,Y)X)$. Then the barrier equation can be written as follows:
\begin{eqnarray}\label{Barrier on non-Einstein manifold}
    \mathcal{B}(b,e_1)&= &-2b^2+2\langle \nabla b, \nabla v\rangle -2\sum_{i=2}^n \kappa_i(v_i^2+v_{ii})-2b\Ric_{11}+\Delta b \\
    & & -\sum_{j}v_j\Ric_{11,j}+2\sum_{j}v_j\Ric_{1j,1},\nonumber 
\end{eqnarray}
where $\kappa_i = \langle R(e_1, e_i, e_1), e_i\rangle$ are sectional curvatures (and eigenvalues of $R_X$) and we write $v_{ii} = \nabla ^2 v(e_i,e_i)$ as well as $v_i = \nabla v \cdot e_i$. We start by simplifying the curvature terms in \eqref{Barrier on non-Einstein manifold}. \begin{align}
      -2\sum_{i = 2}^n \kappa_i(v_i^2+v_{ii})
      & =-2\sum_{i = 2}^n \underline\kappa(v_i^2+v_{ii})  -2\sum_{i = 2}^n (\kappa_i-\underline\kappa)(v_i^2+v_{ii}) \nonumber\\
      & \geq  -2\underline{\kappa}{}( -\lambda - v_1^2 - v_{11})  -2\sum_{i = 2}^n (\kappa_i-\underline\kappa)(v_i^2-b) \nonumber \\
      &  \geq  2\underline{\kappa}{}( \lambda + v_1^2 - b)  + 2\left(\Ric_{11}-(n-1)\underline \kappa\right)  b-2P_\kappa |\nabla v|^2, \label{ineq: trace-term}
\end{align} 
where we used \eqref{Delta-w} in the first inequality and write 
$P_\kappa = \overline \kappa - \underline \kappa.$\footnote{In a space form, this term simplifies since $\Ric_{XX} \equiv (n-1)\underline \kappa$ and $P_\kappa(x) \equiv 0.$}. 
Moreover, for $\varepsilon > 0$ Cauchy-Schwarz and Young's inequality implies that
 \begin{equation} \label{Ricci derivative bound}
    -\sum_{j}v_j\Ric_{11,j}+2\sum_{j}v_j\Ric_{1j,1} \geq -3 |\nabla \Ric | _\infty | \nabla v|  \geq - \frac{9}{4\varepsilon} | \nabla \Ric |^2 -  \varepsilon |\nabla v|^2.
 \end{equation}
We thus have, incorporating \eqref{ineq: trace-term} and \eqref{Ricci derivative bound} into \eqref{Barrier on non-Einstein manifold}\begin{align}
\nonumber \mathcal{B}(b, e_1)
\geq &  -2b^2 +  2\langle \nabla v, \nabla b\rangle  +2\underline{\kappa}{}\lambda   + 2\underline{\kappa} v_1^2 +2\left(\Ric_{11}-(n+1)\underline \kappa\right)  b-2P_\kappa |\nabla v|^2\\
&  -2 \Ric_{11}b + \Delta b - \frac{9}{4\varepsilon} | \nabla \Ric |^2 -  \varepsilon |\nabla v|^2\nonumber\\
=&  -2b^2 +  2\langle \nabla v, \nabla b\rangle + \Delta b +2\underline{\kappa}{} \lambda  +2\underline{\kappa} v_1^2-2(n+ 1)\underline \kappa b-(2P_\kappa +\varepsilon)|\nabla v|^2- \tfrac{9}{4\varepsilon} | \nabla \Ric |^2. \label{ineq: remaining terms to estimate}
\end{align}

Next, we estimate the terms $-b^2  + 2\langle \nabla b, \nabla v\rangle + \Delta b.$ To do so, Bochner's Formula shows that \begin{eqnarray}\label{Incorporating Bochner's formula}
    \Delta |\nabla w|^2 &= & 2 |\Hess \, w|^2 + 2\Ric (\nabla w,\nabla w)+2\langle \nabla w, \nabla \Delta w\rangle 
    \\
    &= & 2\sum_{i,j}w_{ij}^2+2\sum_{i,j}\Ric _{ij}w_iw_j +2\left(-\sum_{i,j} 2 \frac{w_{ij} w_i w_j}{w} + \frac{|\nabla w|^4}{w^2} - \frac{1}{2}  \lambda |\nabla w|^2\right). \nonumber
\end{eqnarray}
In addition, we have that
\begin{align} \label{Gradient inner product}
2\langle \nabla b , \nabla v \rangle = 2\left(\sum_{i,j} 4 \alpha \frac{w_{ij} w_i w_j}{w}  + {  C|\nabla v|^2}\right).
\end{align}
Combining \eqref{Delta-w}, \eqref{Gradient inner product} and \eqref{Incorporating Bochner's formula}, we find that
\begin{eqnarray} \label{Combined Laplacian and gradient inner product}
    \Delta b + 2\langle \nabla b, \nabla v\rangle &= & 2\alpha\sum_{i,j}w_{ij}^2+2\alpha\sum_{ij}\Ric _{ij}w_iw_j  \\
     & &+2\alpha\left(\sum_{i,j} 2 \frac{w_{ij} w_i w_j}{w} + \frac{|\nabla w|^4}{w^2} - \frac{1}{2}  \lambda |\nabla w|^2 \right) + C|\nabla v|^2-C\lambda\nonumber
\end{eqnarray}

Next, we observe that three of the terms in \eqref{Combined Laplacian and gradient inner product} combine to form a square:
\begin{align}\label{Completing the square}
    2 \alpha \sum_{i,j} w_{ij}^2 + 4 \alpha \sum_{i,j}  \frac{w_{ij} w_i w_j}{w}  +2\alpha  \frac{|\nabla w|^4}{w^2}  =2\alpha  \Bigl|\Hess w + \tfrac{1}{w}\nabla w \otimes \nabla w \Bigr|^2.
\end{align}
This is positive,\footnote{Using the Cauchy-Schwarz inequality, it is possible to show that this term exceeds $ \frac{2 \alpha}{n} \lambda^2 w^2$.} so dropping the left-hand side of \eqref{Completing the square} from \eqref{Combined Laplacian and gradient inner product}, we find that
    \begin{align}
\Delta b + 2\langle \nabla b, \nabla v\rangle >  & 2\alpha\sum_{ij}\Ric _{ij}w_iw_j -\alpha \lambda |\nabla w|^2 + C|\nabla v|^2 -C\lambda. \label{eqn: derivatives on b}
\end{align}

Using our choice of barrier \eqref{The barrier involving w}, we find that 
\begin{align}
        -2b^2 = -2(\alpha |\nabla w|^2 +v-d)^2 &= -2\alpha ^2|\nabla w|^4-2(Cv-d)\left(2\alpha |\nabla w|^2 +Cv-d\right) \nonumber \\
        &\geq -2\alpha ^2|\nabla w|^4 + 2 (Cv-d)^2,\label{ineq: square-term}
\end{align}
where we used the assumption that $v$ is log-concave and $v_{11}=-\alpha |\nabla w|^2 -Cv+d \leq 0 $ at $p$ to observe that $2\alpha |\nabla w|^2 \geq-2(Cv-d)$.

Incorporating \eqref{eqn: derivatives on b}, \eqref{ineq: square-term} into \eqref{ineq: remaining terms to estimate}, after rearranging the equation we find that
\begin{align} 
     \mathcal{B}(b,e_1)\geq &   -2\alpha^2 |\nabla w|^4  + 2 (Cv-d)^2   +2\alpha\sum\Ric _{ij}w_iw_j -\alpha \lambda |\nabla w|^2 + C|\nabla v|^2 -C\lambda \nonumber\\
 &\quad  +2\underline{\kappa}{} \lambda +2\underline{\kappa} v_1^2  -2(n+ 1)\underline \kappa b-(2P_\kappa +\varepsilon)|\nabla v|^2 - \tfrac{9}{4\varepsilon} | \nabla \Ric |^2 \nonumber\\
 = & -2\alpha^2 |\nabla w|^4  + 2 (Cv-d)^2   +2\alpha\sum\Ric _{ij}w_iw_j -\alpha \lambda |\nabla w|^2 +(C-2P_\kappa -\varepsilon)|\nabla v|^2 \nonumber \\
 &\quad  +(2\underline{\kappa}-C)\lambda +2\underline{\kappa} v_1^2  -2(n+ 1)\underline \kappa \left(\alpha |\nabla w|^2 + Cv -d\right)- \tfrac{9}{4\varepsilon} | \nabla \Ric |^2.  \label{ineq: last-step}
\end{align}
Finally, we estimate that 
\begin{equation} 
\label{Nearly Einstein identity}
   -2 (n+1) \underline{\kappa}(\alpha |\nabla w|^2 +Cv-d) +\sum_{i, j}2 \alpha \Ric_{ij} w_i w_j
     \geq  -2 (n + 1) \underline{\kappa}(Cv-d)-2\alpha \underline \kappa |\nabla w |^2.
\end{equation}

It is now necessary to consider the case where $\underline{\kappa} \geq 0$ and $\underline{\kappa} <0$ separately. When $\underline{\kappa} \geq 0$, we can drop the term $2\underline{\kappa} v_1^2$ from \eqref{ineq: last-step}. Then, using \eqref{Nearly Einstein identity} and grouping the this expression into gradient terms (i.e., those containing either $|\nabla w|$ or $|\nabla v|$) and non-gradient terms, we obtain

\begin{align} \label{Barrier after grouping}
     \mathcal{B}(b,e_1)& \geq  
 \underbrace{-2\alpha^2|\nabla w|^4 + \left(C-\varepsilon - 2 P_{\kappa} \right)|\nabla v|^2-\alpha ( \lambda +2\underline{\kappa}) |\nabla w|^2}_{\textrm{Gradient terms}} \\
 & + \underbrace{2 (Cv-d)^2  - 2(Cv-d)(n+1)\underline \kappa  + (2 \underline{\kappa}-C)  \lambda - \frac{9}{4\varepsilon} | \nabla \Ric |^2}_{\textrm{Non-gradient terms}}. \label{Non-Gradient terms}
\end{align}

From the fact that $\sup u =1$, the gradient terms may be estimated as
\begin{align}\nonumber
    &-2\alpha^2|\nabla w|^4 + \left(C-\varepsilon - 2 P_{\kappa} \right)|\nabla v|^2-\alpha ( \lambda +2\underline{\kappa}) |\nabla w|^2\\
    &\geq  \left(-\frac{\alpha^2}{2}|\nabla u|^2-\left( \lambda \nonumber+2\underline{\kappa}\right)\alpha+ 4\left(C-\varepsilon - 2 P_{\kappa} \right)\right)\frac{|\nabla u|^2}{4u^2}
    \\
    &\geq  \left(-\frac{\alpha^2}{2}\lambda-\left( \lambda +2\underline{\kappa}\right)\alpha+ 4\left(C-\varepsilon - 2 P_{\kappa} \right)\right)\frac{|\nabla u|^2}{4u^2},\label{alpha-ineq}
\end{align}
where we used \eqref{Gradient bound from Ling} in the last inequality.  Solving this resulting quadratic expression for $\alpha, $ we get that 
\begin{equation} \label{Interior bound on alpha 3}
    \alpha \leq  \frac{  \sqrt{  \left(\lambda+2 \underline{\kappa} \right)^2  +  8 \lambda (C-\varepsilon - 2 P_\kappa) } - (\lambda+2 \underline{\kappa}) }{\lambda},
\end{equation}
and thus $C > 2 P_\kappa + \varepsilon.$
To show that the non-gradient terms are positive, note that $Cv$ is less than $0$ and so we have that
\begin{align*}
    \nonumber
   & 2 (Cv-d)^2  - 2(Cv-d)(n+1)\underline \kappa  + (2 \underline{\kappa}-C)  \lambda - \tfrac{9}{4\varepsilon} | \nabla \Ric |^2\\
    \geq &   2 d^2 + 2(n+1)\underline \kappa d + (2 \underline{\kappa}-C)  \lambda - \tfrac{9}{4\varepsilon} | \nabla \Ric |^2.
\end{align*}
Solving the quadratic for $d$ (and recalling that $d$ is non-negative), we find that this is non-negative whenever
\begin{align}\label{condition on d}
    d \geq \max\left\{ \frac{1}{2} \left(-(n+1)\underline \kappa + \sqrt{\left( (n+1)\underline \kappa\right)^2 - (4 \underline{\kappa}-2C)  \lambda + \tfrac{9}{2\varepsilon} | \nabla \Ric |^2 } \right), 0\right\}.
\end{align}

In the case where $\underline{\kappa}<0,$ we use the inequality $v_1^2 < |\nabla v|^2$ to rewrite the gradient terms as
\[-2\alpha^2|\nabla w|^4 + \left(C-\varepsilon + 2\underline{\kappa} - 2 P_{\kappa} \right)|\nabla v|^2-\alpha ( \lambda +2\underline{\kappa}) |\nabla w|^2. \]

Doing so, we find the bound
\begin{equation} \label{Interior bound on alpha 4}
    \alpha \leq  \frac{  \sqrt{  \left(\lambda+2 \underline{\kappa}\right)^2  +  8 |\nabla u|^2_\infty (C-\varepsilon - 2 P_\kappa+ 2\underline \kappa) } - (\lambda+2 \underline{\kappa})}{| \nabla u |_\infty^2},
\end{equation}
    which forces $C > 2 P_\kappa - 2 \underline{\kappa} + \varepsilon.$

The non-gradient terms can then be estimated as follows:\footnote{It is possible to sharpen this estimate to get a smaller value for $d$.}
\begin{align}
& 2 (Cv-d)^2  - 2(Cv-d)(n+1)\underline \kappa  + (2 \underline{\kappa}-C)  \lambda - \tfrac{9}{4\varepsilon} | \nabla \Ric |^2\\
  \geq & d^2  + (2 \underline{\kappa}-C)  \lambda - \tfrac{9}{4\varepsilon} | \nabla \Ric |^2, \nonumber
\end{align}
which is non-negative whenever
\begin{align}\label{d in negative curvature}
    d \geq \sqrt{ (C - 2 \underline{\kappa})  \lambda + \tfrac{9}{4\varepsilon} | \nabla \Ric |^2}.
\end{align}

\end{proof}
\begin{remark}
    Creating the squared term in \eqref{Completing the square} is the most delicate part of the proof. In particular, this step does not go through if we set $w = u^p$ for $p> \frac{1}{2}$. Furthermore, for $p<\frac{1}{2}$, Lemma \ref{Lemma at boundary, second attempt} fails which forces the choice of $w = \sqrt{u}$.
\end{remark}

\subsection{The continuity family}

We are now able to prove Theorem \ref{Log-concavity-estimate-sphere} by constructing a continuity family.

\begin{proof}[Proof of Theorem \ref{Log-concavity-estimate-sphere}]
    In order to apply Lemma \ref{mainpropposition} with the barrier from Lemma \ref{Lemma to show barrier property}, we need to find a valid continuity family. To do so, we fix the domain $\Omega$ and ambient geometry and deform the barrier by
    \[ b(x,t) = t |\nabla w|^2 + Cv - d\] as $t$ varies from $0$ to $\alpha$. 

When $t=0$, Lemma \ref{Lemma at boundary, second attempt} shows that inequality \eqref{Boundary inequality} does not fail immediately after the initial time. Therefore, so long as $\alpha$ is sufficiently small (as determined by \eqref{Interior bound on alpha 3} and \eqref{Boundary inequality}), the previous lemma proves that the inequality holds throughout the domain. Note that independent of the choices of $C, \varepsilon$ in \eqref{Interior bound on alpha 3}, Lemma \ref{Lemma at boundary, second attempt} forces \begin{equation}\label{bound on alpha from boundary lemma}
    \alpha \leq 4 \min_{X_0 \in U\partial \Omega}  \frac{ \textrm{II}_{x_0}(X_0,X_0) }{ |\nabla u(x_0)|}.
\end{equation} 
\end{proof}

\section{Applications of Theorem \ref{Log-concavity-estimate-sphere}}
\label{Constant curvature section}

We now derive Corollaries \ref{Spherical version of Bochner theorem} and \ref{Euclidean version of Bochner theorem}, respectively. Note that in these geometries, it is known that the principal eigenfunction on convex regions is log-concave (see \cite{brascamp1976extensions} for the Euclidean case and \cite{lee1987estimate} for the spherical case).

\begin{proof}[Proof of Corollary \ref{Spherical version of Bochner theorem}]
    Within Lemma \ref{Lemma to show barrier property}, it was necessary to take $\alpha$ small enough so \eqref{Interior bound on alpha 3} holds. For a sphere, $ P_{\kappa} \equiv0$, so we can take $\varepsilon =0$ and $\underline{\kappa} = 1$. As such, Equation \eqref{Interior bound on alpha 3} simplifies to the following: 
\begin{equation}\label{Bound on alpha on sphere}
    \alpha \leq \frac{1}{\lambda} \left( \sqrt{ (\lambda+2)^2 + 8 \lambda} - (\lambda+2) \right).
\end{equation}
We must also ensure that $\alpha$ is small enough so that \eqref{Assumption near x naught} holds. When all the eigenvalues of the second fundamental form of the boundary are at least $\frac{1}{3}$, \eqref{Gradient bound from Ling} implies that

\[ \frac{4 \textrm{II}(X,X)}{|\nabla  u(x)|} \geq \frac{4}{3 \sqrt{\lambda}},\]
so it suffices to take $\alpha < \frac{4}{3\sqrt{\lambda}}$. Note that for convex domains in $\mathbb{S}^n$, $\lambda$ is at least $n$ (i.e. the principal eigenvalue of a hemisphere). A computation shows that  $\frac{4}{3\sqrt{\lambda}} > \frac{1}{\lambda} \left( \sqrt{ (\lambda+2)^2 + 8 \lambda} - (2 + \lambda) \right)$ for all $\lambda>n,$ so we only must consider the first inequality. As a result, we can take $\alpha =  \frac{1}{\lambda} \left( \sqrt{ (\lambda+2)^2 + 8 \lambda} - (\lambda+2) \right)$. Finally, we can drop the assumption that the domain is $C^2$ by taking a sequence of smooth convex domains approximating $\Omega$ and observing that if the log-concavity estimate holds for all the domains in this sequence, it must hold in the limit as well.
\end{proof}

We now turn our attention to convex domains in Euclidean space. However, due to the vanishing curvature, the estimate given in Corollary \ref{Euclidean version of Bochner theorem} is slightly weaker.
 \begin{proof}[Proof of Corollary \ref{Euclidean version of Bochner theorem}]
 Since in this case $\underline \kappa = 0$. Equation \eqref{condition on d} simplifies to 
\begin{equation}
        d \geq \sqrt{\frac{C\lambda}{2}},
\end{equation}
as well as \eqref{Interior bound on alpha 3} becomes
\[\alpha \leq \frac{1}{\lambda} \left(\sqrt{ \lambda^2 + 8C \lambda} - \lambda \right), \]
where in both inequalities we can choose $C = 1.$
In view of Lemma \ref{Lemma at boundary, second attempt}, we must have that
\[\alpha <  \inf_{X \in U \partial \Omega} \frac{4 \textrm{II}(X,X)}{|\nabla u(x)|}.  \]

 \end{proof}

One natural question is what happens when $\Omega$ is not smooth. To answer this, it is instructive to consider the case when $\Omega \subset \mathbb{R}^2$ is a square. In this case, a straightforward computation shows that for any value of $\alpha$, the quantity $\alpha | \nabla \sqrt u  |^2 + v$ tends to $- \infty$ along certain sequences converging to each corner.\footnote{Due to the restriction imposed by the second fundamental form, Corollary \ref{Euclidean version of Bochner theorem} does not hold for the square for any $\alpha >0$. Nevertheless, this example is useful for understanding the behavior near the boundary.} In particular, the Hopf lemma does not hold near corners, so even though Corollary \ref{Euclidean version of Bochner theorem} does not require smoothness of the domain, the resulting estimate will be vacuous near corners.

 \subsection{Other geometries}

We now apply Theorem \ref{Log-concavity-estimate-sphere} in other spaces of non-constant sectional curvature. For these, the assumption that the first eigenfunction is log-concave becomes important, since this might be unknown. For example, for $C^3$ perturbations of the round sphere (for $n \geq 3$), we have the following. 
 \begin{theorem}
    Suppose $(M^n,g)$ is a Riemannian manifold satisfying $\underline{\kappa} = 1$, $P_\kappa<\frac{1}{4}$ and $|\nabla \Ric|< \frac{n+1}{9}$. Then for any uniformly convex $\Omega$, if $v = \log u$ is concave, one has that 
    \begin{align*}
        \nabla ^2 v + \alpha |\nabla \sqrt u|^2 + v <0.
    \end{align*}
    Here $\alpha$ depends on $\lambda , $ $P_\kappa,$ $ |\nabla \Ric|,$ and the second fundamental form of $\partial \Omega.$
\end{theorem}
\begin{proof}
    This follows from choosing $C =1 $ in Theorem \ref{Log-concavity-estimate-sphere} and choosing $\varepsilon = \tfrac{1}{2}$ in \eqref{Barrier after grouping} and \eqref{Non-Gradient terms}.
\end{proof}

Similarly, for positively curved Einstein manifolds, we have the following.

\begin{obs} \label{Complex projective space estimate}
    Let $\Omega$ be a convex domain in $(M,g),$ where $\overline \kappa \geq \kappa \geq \underline \kappa \geq  0$ and $\nabla \Ric \equiv 0$ (e.g. $g$ is Einstein). If $v= \log u$ is a concave function, it also satisfies the estimate
    \begin{equation*}
        \nabla^2 v + \alpha |\nabla \sqrt{u}|^2  + \left(2P_\kappa+1\right)v < \max\left\{\tfrac{-(n+1)\underline \kappa}{2} + \sqrt{\left( \tfrac{(n+1)\underline \kappa}{2}\right)^2 - \left( \underline{\kappa}-P_\kappa - \tfrac{1}{4}\right)  \lambda } , 0 \right\}.
    \end{equation*}
    where $\alpha$ depends on $\lambda,$ $ P_\kappa$ and the second fundamental form of $\partial \Omega$.\end{obs}

\begin{proof}
This follows from choosing $ C = 2P_\kappa + 1$ and $\varepsilon = 0 $. Then, we must choose $\alpha$ to satisfy the inequalities \eqref{bound on alpha from boundary lemma} and \eqref{Interior bound on alpha 3}.
\end{proof}

We can obtain a similar estimate in hyperbolic space, although here it is necessary to assume that the eigenfunction is log-concave, since this can fail for convex regions \cite{bourni2022vanishing}.

\begin{obs} \label{Hyperbolic space estimate}
    Let $\Omega$ be a convex domain in $\mathbb{H}^n$. If $v= \log u$ is a concave function, it also satisfies the estimate
    \begin{equation}
        \nabla^2 v + \alpha |\nabla \sqrt{u}|^2  +3v < \frac{1}{2} \left( \sqrt{ 10 \lambda + (n+1)^2} - (n+1) \right).
    \end{equation}
    where $\alpha$ depends on $\lambda$ and the second fundamental form of $\partial \Omega$.\end{obs}

\subsection{Final Remarks} By choosing constants appropriately, it is possible to use this method to obtain bounds which are sharp near the boundary (up to lower order terms). In particular, if we take $C$ large enough so that the right-hand-side of \eqref{Interior bound on alpha 3} exceeds the bound of Lemma \ref{Lemma at boundary, second attempt}, one can take
\begin{equation} \label{Optimal alpha}
    \alpha = 4 \min_{X_0 \in U\partial \Omega}  \frac{ \textrm{II}_{x_0}(X_0,X_0) }{ |\nabla u(x_0)|},
\end{equation} 
which is the value of $\alpha$ which saturates \eqref{Assumption near x naught}.
For this choice of $C, \alpha$ (and corresponding $d$ in \eqref{condition on d}), we find that $\nabla ^2 v + \alpha |\nabla \sqrt u|^2 + Cv -d <0.$
 This provides an estimate whose growth is asymptotically sharp for points and vectors $(x,X) \in U\Omega$ which are close to the $(x_0,X_0) \in U \partial \Omega$ which achieves the minimum in \eqref{Optimal alpha}.

Finally, we would like to note that our results are related to a conjecture by Ishige, Salani and Takatsu \cite{ishige2020new}. They conjectured that for convex $\Omega \subset \mathbb R^n$ the $L^\infty$-normalized ground state eigenfunction is $\tfrac{1}{2}$-log-concave (that is, that $-\sqrt{-\log u}$ is concave). This is stronger than log-concavity and is equivalent to 
\begin{equation} \label{half log-concavity}
      \nabla^2 v - \frac{\nabla v \otimes \nabla v}{2 v} \leq 0
\end{equation}
for all unit vectors $v$ (recall that $v \leq 0$). Corollary \ref{Euclidean version of Bochner theorem} can be written as
    \begin{equation} \label{Euclidean inequality in terms of v}
      \nabla^2 v + \frac{\alpha}{4} \exp(v) |\nabla v|^2 + v \leq  \sqrt{\tfrac{\lambda}{2}}.
\end{equation}
For vectors $X \perp \nabla u$, the latter inequality is actually stronger because it includes the norm of the gradient rather than the norm of the directional derivative. However, this inequality is weaker in the gradient direction (and includes lower-order terms). Nevertheless, one can consider Corollary \ref{Euclidean version of Bochner theorem} as providing partial evidence toward this conjecture.

\bibliography{references}

\newcommand{\etalchar}[1]{$^{#1}$}
\begin{thebibliography}{KNTW22}

\bibitem[AC11]{andrews2011proof}
Ben Andrews and Julie Clutterbuck.
\newblock Proof of the fundamental gap conjecture.
\newblock {\em Journal of the American Mathematical Society}, 24(3):899--916,
  2011.

\bibitem[BCN{\etalchar{+}}22]{bourni2022vanishing}
Theodora Bourni, Julie Clutterbuck, Xuan~Hien Nguyen, Alina Stancu, Guofang
  Wei, and Valentina-Mira Wheeler.
\newblock The vanishing of the fundamental gap of convex domains in {$\mathbb
  {H}^n$}.
\newblock {\em Ann. Henri Poincar\'{e}}, 23(2):595--614, 2022.

\bibitem[BL76]{brascamp1976extensions}
Herm~Jan Brascamp and Elliott~H Lieb.
\newblock On extensions of the {B}runn-{M}inkowski and {P}r{\'e}kopa-{L}eindler
  theorems, including inequalities for log concave functions, and with an
  application to the diffusion equation.
\newblock {\em Journal of functional analysis}, 22(4):366--389, 1976.

\bibitem[GT77]{gilbarg1977elliptic}
David Gilbarg and Neil~S Trudinger.
\newblock {\em Elliptic partial differential equations of second order}.
\newblock Number~2. Springer, 1977.

\bibitem[IST20]{ishige2020new}
Kazuhiro Ishige, Paolo Salani, and Asuka Takatsu.
\newblock New characterizations of log-concavity.
\newblock {\em arXiv preprint arXiv:2004.13381}, 2020.

\bibitem[IST22]{Ishige2022}
Kazuhiro Ishige, Paolo Salani, and Asuka Takatsu.
\newblock Power concavity for elliptic and parabolic boundary value problems on
  rotationally symmetric domains.
\newblock {\em Commun. Contemp. Math.}, 24(9):Paper No. 2150097, 29, 2022.

\bibitem[KN24]{khan2022negative}
Gabriel Khan and Xuan~Hien Nguyen.
\newblock Negative curvature constricts the fundamental gap of convex domains.
\newblock {\em Ann. Henri Poincar\'{e}}, 2024.

\bibitem[KNTW22]{surfacepaper1}
Gabriel Khan, Xuan~Hien Nguyen, Malik Tuerkoen, and Guofang Wei.
\newblock Log-concavity and fundamental gaps on surfaces of positive curvature.
\newblock {\em To appear in Communications in Analysis and Geometry,
  arXiv:2211.06403}, 2022.

\bibitem[KST24]{khan2024concavity}
Gabriel Khan, Soumyajit Saha, and Malik Tuerkoen.
\newblock Concavity properties of solutions of elliptic equations under
  conformal deformations.
\newblock {\em arXiv preprint arXiv:2403.03200}, 2024.

\bibitem[KT24]{khan2024spectral}
Gabriel Khan and Malik Tuerkoen.
\newblock Spectral gap estimates on conformally flat manifolds.
\newblock {\em arXiv preprint arXiv:2404.15645}, 2024.

\bibitem[KTW23]{khan2023modulus}
Gabriel Khan, Malik Tuerkoen, and Guofang Wei.
\newblock Modulus of concavity and fundamental gap estimates on surfaces.
\newblock {\em arXiv preprint arXiv:2306.06053}, 2023.

\bibitem[Lin06]{ling2006lower}
Jun Ling.
\newblock A lower bound of the first {D}irichlet eigenvalue of a compact
  manifold with positive {R}icci curvature.
\newblock {\em International Journal of Mathematics}, 17(05):605--617, 2006.

\bibitem[LW87]{lee1987estimate}
Yng-Ing Lee and Ai~Nung Wang.
\newblock Estimate of $\lambda_2 - \lambda_1$ on spheres.
\newblock {\em Chinese Journal of Mathematics}, pages 95--97, 1987.

\bibitem[SWW19]{10.4310/jdg/1559786428}
Shoo Seto, Lili Wang, and Guofang Wei.
\newblock {Sharp fundamental gap estimate on convex domains of sphere}.
\newblock {\em Journal of Differential Geometry}, 112(2):347 -- 389, 2019.

\bibitem[SWYY85]{singer1985estimate}
I.M. Singer, Bun Wong, Shing-Tung Yau, and Stephen S.-T. Yau.
\newblock An estimate of the gap of the first two eigenvalues in the
  {S}chr{\"o}dinger operator.
\newblock {\em Annali della Scuola Normale Superiore di Pisa-Classe di
  Scienze}, 12(2):319--333, 1985.

\bibitem[Wan00]{wang2000estimation}
F-Y Wang.
\newblock On estimation of the {D}irichlet spectral gap.
\newblock {\em Archiv der Mathematik}, 75:450--455, 2000.

\end{thebibliography}
\bibliographystyle{alpha}

\end{document}